\documentclass[12pt]{amsart}

\usepackage{amsmath,amsthm,amssymb,amscd,url,enumerate,mathtools}
\usepackage{graphicx}
\usepackage[margin=1in]{geometry}
\usepackage{afterpage}
\usepackage{tikz} \usetikzlibrary{cd} 
\usepackage[loadshadowlibrary]{todonotes}
\usepackage[all]{xy}
\usepackage[colorlinks=true,citecolor={blue},linkcolor = {purple}]{hyperref} 
\usepackage{tabularx}

\usepackage{caption}

\usepackage[square]{natbib}
\setcitestyle{numbers}

\theoremstyle{plain}
\newtheorem{theorem}{Theorem}

\newtheorem{proposition}[theorem]{Proposition}
\newtheorem{lemma}[theorem]{Lemma}
\newtheorem{corollary}[theorem]{Corollary}

\theoremstyle{definition}
\newtheorem{definition}[theorem]{Definition}

\theoremstyle{remark}
\newtheorem{remark}[theorem]{Remark}

\numberwithin{theorem}{section}

  {\end{list}}

\newcommand{\tightoverset}[2]{%
  \mathop{#2}\limits^{\vbox to -.5ex{\kern-1.05ex\hbox{$#1$}\vss}}}

\numberwithin{equation}{section} 

\newcommand{\CC}{\mathbb{C}}

\newcommand{\ZZ}{\mathbb{Z}}

\def \bfr{{\mathbf r}}

\def \bfv{{\mathbf v}}

\newcommand{\defeq}{\stackrel{\mathrm{def}}{\, = \,}}

\newcommand{\bZ}{{\mathbb{Z}}}
\newcommand{\bR}{{\mathbb{R}}}
\newcommand{\bC}{{\mathbb{C}}}
\newcommand{\bQ}{{\mathbb{Q}}}

\newcommand{\mc}{\mathcal} 
\newcommand{\mb}{\mathbf} 

\newcommand{\torus}{{\mathbb{T}}}

\newcommand{\TITLE}{General Boyd--Lawton Theorems with Multivariable Limits}

\title {\TITLE}

\date{\today}

\author[Aitken]{Wayne Aitken}
\address{Department of Mathematics\\
California State University San Marcos\\
333 S. Twin Oaks Valley Rd.\\
San Marcos, CA 92096}
\email{waitken@csusm.edu}

\author[Ayers]{Kimberly Ayers}
\address{Department of Mathematics\\
California State University San Marcos\\
333 S. Twin Oaks Valley Rd.\\
San Marcos, CA 92096}
\email{kayers@csusm.edu}

\author[Smith]{Hanson Smith}
\address{Department of Mathematics\\
California State University San Marcos\\
333 S. Twin Oaks Valley Rd.\\
San Marcos, CA 92096}
\email{hsmith@csusm.edu}


\begin{document}

\begin{abstract}
The classical Boyd--Lawton theorem concerning
Mahler measures has recently been
extended to multivariable limits by Brunault, Guilloux, Mehrabdollahei, and Pengo.
In another direction, the single-variable Boyd--Lawton theorem
has been generalized to various extensions of Mahler measure by Issa and Lal\'in.
The goal of this paper is to present a cohesive framework for extending single-variable Boyd--Lawton theorems
to multivariable Boyd--Lawton theorems. With this, we broaden the single-variable Boyd--Lawton theorems of Issa and Lal\'in
to multivariable versions in the direction of Brunault, Guilloux, Mehrabdollahei, and Pengo, providing a generalization of both works.

\end{abstract}

\maketitle


\section{Introduction} \label{intro}

The classical Boyd--Lawton theorem of  \citep{boyd1981} and \citep{lawton1983} concerning
Mahler measures has been
extended to multivariable limits in
\cite{Brunault2024}, while the single-variable Boyd--Lawton theorem
has been extended to various generalizations of the Mahler measure (see~\cite[Theorem~1.2]{Issa2013}).
The goal of this paper is to present a straightforward framework for generalizing single-variable Boyd--Lawton theorems
to multivariable Boyd--Lawton theorems. Using this framework, we extend Boyd--Lawton theorems such as Theorem 1.2 of~\cite{Issa2013} to multivariable limits along the lines of Theorem~1.1 of~\cite{Brunault2024}.
Along the way, our result yields an independent proof of Theorem~1.1 of \cite{Brunault2024} via a separate line of reasoning.

We start with a quick review of traditional Mahler measure and
the associated Boyd--Lawton theorem.
Mahler measure is given by an integral on the $n$-dimension unit torus:
$$
\torus^n = \{\; (z_1, \ldots, z_n) \in \bC^n \; \colon 
	 \; |z_1| =|z_2| =  \cdots = |z_n| = 1 \}.
$$
We regard $\torus^n$ as a compact real Lie group under coordinate-wise multiplication.
Suppose that~$P \in \bC\left[Z^{\pm 1}_1, \ldots, Z^{\pm 1}_n\right]$ is a nonzero Laurent polynomial.\footnote{Throughout this paper we will engage in a slight abuse of
notation by using the same symbol $P$ for a Laurent polynomial, regarded as an algebraic object, and for the 
associated function $\torus^n \to \bC$ defined by
the rule~$
(z_1, \ldots, z_n) \mapsto P(z_1. \ldots, z_n)$.  
Context should help distinguish usage, which is aided by our use
of  upper-case variables such as~$Z_i$ for  purely symbolic variables and lower-case letters such as~$z_i$ for corresponding complex variables that usually vary on the unit circle~$\torus$.}
Then 
the $n$-variable \emph{Mahler measure}~$m_n(P)$ of $P$ is defined to be the average
value of the function $\log |P|$ on the unit torus~$\torus^n$:
$$
m_n(P) \; \defeq \; \int_{\torus^n} \log |P| \, d\tau_n  = \int_{0}^{1} \cdots \int_{0}^{1}  \log
\left| P\left(e^{2\pi t_1 \mb i}, \ldots, e^{2\pi t_n \mb i }\right) \right|
\;  d t_1 d t_2 \cdots d t_n
$$
where
$\tau_n$ denotes the normalized Haar measure on $\torus^n$.
In the special case of $n=1$, we write~$\torus$ for the unit circle~$\torus^1$ in $\bC$
and write~$\tau$ for the normalized Haar measure $\tau_1$.

The Mahler measure was considered by Mahler~\cite{Mahler}, but the single variable version was considered
much earlier by Lehmer \cite{Lehmer} in the context of constructing large prime numbers.

The famous \emph{Boyd--Lawton theorem} expresses any multivariable
Mahler measure as a limit of single variable Mahler measures.  
Suppose $P \in \bC\left[Z_1^{\pm 1}, \ldots, Z^{\pm 1}_n\right]$ is a nonzero Laurent polynomial where $n \ge 2$.
For each $\mb r = (r_1, \ldots, r_n) \in \bZ^n$
define the associated power-substitution $P^{(\mb r)} \in \bC\left[Z^{\pm 1}\right]$ to be the Laurent polynomial defined as follows:
$$
P^{(\mb r)} (Z) \; \defeq \;  P\left( Z^{r_1}, \ldots, Z^{r_n} \right).
$$
Define the \emph{Boyd height} of ${\mb r}$ as follows (cp.~\citep[p.~118]{boyd1981}):
$$
\mu\left( \mb r \right) 
\; \defeq \;  
\min 
\left\{ \| \mb v\|_\infty  \colon
    \text{$\mb v \in \bZ^n, \mb v \ne {\mb 0}$, and $\mb v \cdot \mb r = 0$} 
 \right\}.
 $$
 Since $n \ge 2$, any $\mb r \in \bZ^n$ has a nonzero $\mb v \in \bZ^n$ perpendicular to it. 
 Thus~$\mu\left({\mb r}\right)$ is well-defined
and is  a positive integer.
 
The Boyd--Lawton theorem then asserts the following:
$$
m_n(P) = \lim_{\mu\left({\mb r}\right) \to\infty} m_1\left(P^{(\mb r)}\right)
$$
where ${\mb r}$ varies among lattice points of~$\bZ^n$.

In this paper we will view the  Boyd--Lawton theorem in terms of composition with continuous homomorphisms $\torus \to \torus^n$.
If $\mb r = (r_1, \ldots, r_n)$ then let $q_{\mb r} \colon \torus \to \torus^n$ be the continuous 
homomorphism defined by the equation
$$
q_{\mb r} (z) = \big(z^{r_1}, \ldots, z^{r_n}\big).
$$
We define the \emph{Boyd height} $\mu\left( q_{\mb r} \right)$ of $q_{\mb r} \colon \torus \to \torus^n$ to be $\mu(\mb r)$, where $\mu(\mb r)$ is as above.
(See Section~\ref{height_section}  below for another definition of~$\mu\left(q_{\mb r}\right)$ in terms of homomorphism themselves.)
With this definition, another way to phrase the Boyd--Lawton theorem is that for any function~$P \colon \torus^n \to \bC$ defined by a $n$-variable Laurent polynomial we have
$$
\int_{\torus^n} \log |P| \, d\tau_n =   \lim_{\mu\left(q_{\mb r}\right) \to\infty}  \int_\torus (\log|P|) \circ q_{\mb r}\ d\tau,
$$
where $q_{\mb r}$ varies among continuous homomorphisms $\torus\to \torus^n$.


Observe that the traditional Mahler measure 
applies to averages of functions of the form~$\log |P|$ where~$P$ is given by a single Laurent polynomial.
In recent decades the concept of Mahler measure has been profitably extended in various ways to sequences $P_1, \ldots, P_k$ of Laurent polynomials. 
We mention first the paper~\cite{GonOyanagi2004}
where Gon and Oyanagi consider what they call \emph{generalized Mahler measures}.
Let $P_1, \ldots, P_k \in \bC\left[Z^{\pm 1}_1, \ldots, Z^{\pm 1}_n\right]$ be nonzero Laurent polynomials, and
let $g$ be the function on $\torus^n$ defined by taking the maximum value
of the respective logs:
$$
g_{\mathrm{max}} (z_1, \ldots, z_n)\;  \defeq  \; \max \{ \log|P_1(z_1, \ldots, z_n)|, \ldots, \log|P_k(z_1, \ldots, z_n)| \}.
$$
Gon and Oyanagi  defined the generalized Mahler measure to be the average value of this maximum function:
$$
m_{n, \mathrm{max}} (P_1, \ldots, P_k) \; \defeq \; \int_{\torus^n} g_{\mathrm{max}}  \, d\tau_n.
$$

In \cite{Kurokawa2008} Kurokawa, Lal\'in, and Ochiai considered another extension 
of Mahler measure
which they call \emph{multiple higher Mahler measure}. 
As before let $P_1, \ldots, P_k \in  \bC\left[Z^{\pm 1}_1, \ldots, Z^{\pm 1}_n\right]$ be nonzero Laurent polynomials, 
but this time let
 $g_{\mathrm{prod}}$ be the function on $T^n$ defined by taking the product of logs:
$$
g_{\mathrm{prod}} \;  \defeq  \;  \log|P_1| \cdots \log|P_k|.
$$
Then the \emph{multiple higher Mahler measure} is defined to be
the average value:
$$
m_{n, \mathrm{prod}} (P_1, \ldots, P_k) \; \defeq \; \int_{\torus^n} g_\mathrm{prod} \ d\tau_n.
$$

In \cite{Issa2013}, Issa and Lal\'in show that Boyd--Lawton 
applies to these generalizations of Mahler measure.
Theorem~1.2 of \cite{Issa2013} gives the following limits
$$
m_{n, \textrm{max}} (P_1, \ldots, P_k) = 
\lim_{\mu\left({\mb r}\right) \to\infty} m_{1, \textrm{max}} \left(P_1^{(\mb r)}, \ldots, P_k^{(\mb r)}\right),
$$
and
$$
m_{n, \textrm{prod}} (P_1, \ldots, P_k) = 
\lim_{\mu\left({\mb r}\right) \to\infty} m_{1, \textrm{prod}}\left(P_1^{(\mb r)}, \ldots, P_k^{(\mb r)}\right).
$$
(See also \cite[Theorem 30] {LalinSinha} for the second limit).
In \cite{Issa2013} and \cite{LalinSinha}, $P_1, \ldots, P_k$ are assumed to be polynomials in~$\bC[Z_1, \ldots, Z_n]$, and the coordinates of $\mb r$
are assumed to be positive integers. In Section \ref{extending_Issa} below we will extend the results of \cite{Issa2013} to Laurent polynomials~$P_1, \ldots, P_k$ 
where we allow $\mb r$ to vary among nonzero elements of $\bZ^n$ without sign restriction.
Observe that, so extended, the classical Boyd--Lawton is equivalent to the case where~$k=1$ in either of
these result.

Observe also that all versions of Boyd--Lawton we have considered up to now relate an integral on 
the torus $\torus^n$ to  limits of integrals on the unit circle $\torus$.
\emph{The purpose of the current paper is to show that when we have a class of functions
for which the classic Boyd--Lawton holds (expressing integrals on $\torus^n$ in terms of limits of integrals
on $\torus$), then there is simple way to extend the Boyd--Lawton result to a result
expressing an integral on $\torus^n$ in terms of limits of integrals on $\torus^m$ where 
we allow any positive integer $m$}.
In particular, our methods give a multivariable version of Boyd--Lawton not just for Mahler measures, but also for generalized
Mahler measures and multiple higher Mahler measures.


\subsection{The Main Result}

Our result is partially motivated by a recent result \cite[Theorem~1.1 in their general form given in Theorem 3.1] {Brunault2024} which we can express as follows:
If~$A$ is an $n$ by $m$ matrix with integer coefficients $a_{ij}$
then let~$q_A \colon \torus^m \to \torus^n$
be the continuous homomorphism
defined by the equation
$$
q_A \left(z_1, \ldots, z_m\right) 
=
\left( z_1^{a_{11}}  z_2^{a_{12}} \cdots  z_m^{a_{1m}},  \;  \ldots \; , \; z_1^{a_{n1}}  z_2^{a_{n2}} \cdots  z_m^{a_{nm}} 
\right).
$$
Then for any nonzero $P \in  \bC\left[Z^{\pm 1}_1, \ldots, Z^{\pm 1}_n\right]$ 
\begin{equation} \label{newBL}
\int_{\torus^n}  \log |P| \, d \tau_n  = \lim_{\mu\left(q_A\right) \to \infty}  \int_{\torus^m}  \log |P| \circ q_A \ d \tau_m,
\end{equation}
where
$$
\mu\left(q_A\right) 
\; \defeq \;  
\min 
\left\{ \| \mb v\|_\infty  \colon
    \text{$\mb v \in \bZ^n, \mb v \ne {\mb 0}$, and $\mb v \cdot A = \mb 0$} 
 \right\}
$$
and where in the limit of  (\ref{newBL}) above $A$ varies over the set of integer matrices with $n$ rows, but where the number of 
columns can be any positive integer $m$. (If no such $\mb v$ exists for a particular $A$, then we set $\mu (q_A)$ equal to $\infty$.
Also, note that in~Definition~\ref{boyd_height_definition} we provide a definition of~$\mu\left(q_{A}\right)$ directly in terms of the homomorphism themselves.)

Our result extends the limit (\ref{newBL}) to more general types of functions including those used in the extensions  of Mahler measures considered above.
Our result is expressed  in terms of ``Boyd--Lawton collections''. These are well-behaved collections of functions for which we
already have a single-variable Boyd--Lawton result:

\begin{definition} \label{BL collection def}
A \emph{Boyd--Lawton collection} 
 $\mc C = \bigcup_{n \ge 1} \mc C (\torus^n)$ is a collection of functions such that the following hold:
\begin{enumerate}
\item
If $g \in \mc C (\torus^n)$, then $g$ is defined and is complex-valued almost everywhere on $\torus^n$. Further, each $g \in \mc C(\torus^n)$
is integrable: it gives an element of the $L^1$-space~$L^1(\torus^n)$ where as usual $\torus^n$ is given the normalized Haar measure.
\item
For each $g\in \mc C(\torus^n)$ with $n\ge 2$, a classical Boyd--Lawton theorem holds:
$$
\int_{\torus^n}  g \, d\tau_n= \lim_{\mu(q_\mb r)\to \infty}\;  \int_\torus g \circ q_{\mb r} \, d \tau,
$$
where $q_{\mb r} \colon \torus \to \torus^n$ vary among the continuous homomorphisms $\torus \to\torus^n$ defined by lattice points~$\mb r \in \bZ^n$.
(Implicit here is the requirement that, for large enough $\mu\left(q_\mb r\right)$, the function $g \circ q_{\mb r}$ is integrable.)

\item
The class $\mc C$ is closed under composition by homomorphisms $q_{A}$
in the following sense: If $g$ is in $\mc C(\torus^n)$, then there is a bound $B$ such that if $q_A \colon \torus^m \to \torus^n$ has Boyd-height~$\mu (q_A) \ge B$, 
then the composition  $g\circ q_{A}$ is in $\mc C(\torus^m)$.
\end{enumerate}
\end{definition}

The main theorem of this paper is the following.

\begin{theorem} \label{main_theorem}
Suppose $\mc C$ is a Boyd--Lawton collection of functions. If $g \in \mc C (\torus^n)$,
then
$$
\int_{\torus^{n}} g \, d \tau_{n} \; = \;  \lim_{\mu\left( q_A \right)\to \infty}\;  \int_{\torus^m} g \circ q_A \; d \tau_m,
$$
where $q_A\colon \torus^m \to\torus^n$ varies among continuous homomorphisms. Here $m$ may vary in the limit, but $n$ is fixed by the choice of~$g$.
\end{theorem}

If one prefers to keep $m$ fixed, then the analogous result holds as a corollary. 
Note the
case $n > 1$ and $m=1$ is the single-variable Boyd--Lawton, which here we are assuming for~$g$.

\begin{corollary}
Suppose $\mc C$ is a Boyd--Lawton collection of functions. If $g \in \mc C (\torus^n)$ and if $m$ is a positive integer,
then
$$
\int_{\torus^{n}} g \, d \tau_{n} \; = \;  \lim_{\mu\left( q_A \right)\to \infty}\;  \int_{\torus^m} g \circ q_A \; d \tau_m,
$$
where $q_A$ varies among continuous homomorphisms~$ \torus^m \to\torus^n$.
\end{corollary}

In practice, a reader may prefer the sequential version which aligns better with \cite[Theorem~1.1 and Theorem~3.1] {Brunault2024}:

\begin{corollary}
Suppose  $g \in \mc C(\torus^n)$ where
 $\mc C = \bigcup \mc C(\torus^n)$ is a Boyd--Lawton collection of functions.
 Suppose $(q_{A_i})$ is a sequence of continuous homomorphisms $\torus^{m_i} \to \torus^{n}$
 with each~$m_i\in\ZZ_{>0}$ and the associated sequence of Boyd-heights~$\big( \mu\left( q_{A_i} \right)\big) $ approaching $\infty$. Then
$$
\int_{\torus^{n}} g \, d \tau_{n} 
\; = \; 
\lim_{i\to \infty}  \; \int_{\torus^{m_i}} g \circ q_{A_i} \; d \tau_{m_i}.
$$
\end{corollary}

\bigskip

These results apply to any Boyd--Lawton collections functions including the following:
\begin{enumerate}
\item
The case where $\mc C(\torus^n)$ is the collection of continuous functions $\torus^n \to \bC$.
The single-variable Boyd--Lawton in this case is due to Boyd alone \cite[Lemma~1]{boyd1981}.
\item
The case where $\mc C(\torus^n)$ is the collection of (almost everywhere defined) functions on the unit torus~$\torus^n$ of the form
$$
(z_1, \ldots, z_n) \mapsto \log | P(z_1, \ldots, z_n) |
$$
where $P$ is a nonzero Laurent polynomial.
These are just the function used in defining the traditional Mahler measure.

\item
The case where $\mc C(\torus^n)$ is the collection of (almost everywhere defined) functions on the unit torus~$\torus^n$ of the form
$$
(z_1, \ldots, z_n) \mapsto  \max \left\{ 
    \log | P_1(z_1, \ldots, z_n) |, \ldots, \log| P_k(z_1, \ldots, z_n)| 
    \right\}
$$
where each $P_i$ is a nonzero  Laurent polynomial in the ring  $\bC[Z^{\pm 1}_1, \ldots, Z^{\pm 1}_n]$.
These are the functions related to \emph{generalized Mahler measures} introduced
in~\cite{GonOyanagi2004}.

\item
The case where $\mc C(\torus^n)$ is the collection of (almost everywhere defined) functions on the unit torus~$\torus^n$ of the form
$$
(z_1, \ldots, z_n) \mapsto  
    \log | P_1(z_1, \ldots, z_n) | \cdot \log | P_2(z_1, \ldots, z_n) | \cdots \log| P_k(z_1, \ldots, z_n)| 
$$
where each $P_i$ is a nonzero  Laurent polynomial in the ring  $\bC[Z^{\pm 1}_1, \ldots, Z^{\pm 1}_n]$.
These are the functions related to \emph{multiple higher Mahler measures} of~\cite{Kurokawa2008}.
\end{enumerate}
These collections are considered in more detail in Section~\ref{section_BL_Collections} below and will be shown to be Boyd--Lawton collections.
There are likely other  Boyd--Lawton collections one might wish to consider, extending this already rich set of examples.


\subsection{Related Work and Themes}

There is a good outline of the history and applications related to the Boyd--Lawton theorem
in \cite{Brunault2024} (Section 1.1. Historical remarks). 
The interested reader should also consult
\cite{BertinLalin},
\cite{Brunault_Zudilin_2020}, and~\cite{McKee_Smyth_2021} for  extensive treatments and bibliographies for the general topic of Mahler measures and their applications.  

We have recently become aware of other generalizations of the Mahler measure besides
those discussed above.
For example, H.~Akatsuka~\cite{Akatsuka} considers what are known as zeta Mahler measures, and S.~Roy~\cite{Roy} considers Mahler measures where the integration occurs over other tori than the unit tori.
These generalized Mahler measures are ideal candidates for general  Boyd--Lawton theorems
and seem likely to fall under the framework offered here. Going beyond our framework we note
that partial versions of the Boyd--Lawton theorem have 
been introduced in the more general setting of 
arithmetical dynamics~\cite[Proposition~1.3]{twovariablemahlerPublished}.

One of the most intriguing application of Mahler measure and their generalizations is their connection with special values of 
zeta functions, and $L$-functions more generally.
For example, Boyd \cite{BoydLFunctions} shows that
$$
m_3( Z_1 + Z_2 + Z_3 + 1) = \frac 7{2 \pi^2} \zeta(3)
$$
where $\zeta(s)$ is the Riemann zeta function. 
For generalized Mahler measures, Gon and Oyanagi~\cite{GonOyanagi2004} find formulas such as
$$
m_{4, \mathrm{max}} ( Z_1 + 1, Z_2 + 1, Z_3 + 1, Z_4+1) = \frac 9{2 \pi^2} \zeta(3) - \frac {93}{2\pi^4} \zeta(5).
$$
Kurokawa, Lal\'in, and Ochiai \cite{Kurokawa2008} show results
for
multiple higher Mahler 
measures such as
$$
m_{2, \mathrm{prod}}( Z_1 -1, Z_1-1, Z_1 -1 ) = - \frac {3\zeta(3)}2
$$
and
$$
m_{2, \mathrm{prod}}\left( Z_1 + Z_2 + 2, Z_1 + Z_2 +2, Z_1 + Z_2 + 2\right) = \frac 9 2 \log \big(2 \zeta (2)\big) - \frac  {15} 4 \zeta(3).
$$

In fact, connections with special values of $L$-functions is a major theme of research since the pioneering work in D.~Boyd \cite{BoydSpeculations} and C.~J.~Smyth \cite{Smyth}. The $L$-functions involved include not only the Riemann zeta function and Dirichlet $L$-functions but also $L$-functions of elliptic curves, modular forms, and $K3$-surfaces. This phenomenon has been connected with the Beilinson conjectures since
the work of C.~Deninger \cite{Deninger}.
The connections of certain Mahler measures to $L$-series
of elliptic curves is investigated in many articles starting with~\cite{BoydLFunctions} and~\cite{RVModularMahler}.
Work on the connection of certain Mahler measures to $K3$-surfaces can be found in several articles including~\cite{Bertin2008}.
Interesting connections with $L$-functions of modular forms can be found in~\cite{Samart2013}.
For further information and more citations in this interesting area of research see \cite{BertinLalin} and \cite{Brunault_Zudilin_2020}. 
These deep connections to elliptic curves, K3-surfaces and modular forms are most elaborated for the traditional Mahler measure, but it would be interesting and perhaps useful to see similar connections for the extended Mahler measures considered here.


\section{Continuous Homomorphisms between Real Tori} \label{Section_Homomorphisms}

We now discuss the connection between matrices with entries in $\bZ$ and
continuous homomorphisms $\torus^m \to \torus^n$. 

Throughout this paper $\torus$ will denote the unit circle in $\bC$, and $\torus^n$
will denote the unit torus
$$
\torus^n = \{\; (z_1, \ldots, z_n)  \; \colon 
	 \; |z_1| =|z_2| =  \cdots = |z_n| = 1 \} \subseteq \bC^n.
$$
Note that $\torus^n$ is a multiplicative group under componentwise multiplication.
Its topology is that inherited from the topology of $\bC^n$, and under this
topology $\torus^n$ is a topological group. In fact, $\torus^n$ can be given the structure
of a  real Lie group. The group $\torus^n$ is compact and has a unique normalized
Haar measure $\tau_n$ such that $\torus^n$ has total measure $1$.
This measure is just the~$n$-term product measure $\tau \otimes \cdots \otimes \tau$ formed from $\tau = \tau_1$.
Here $\tau$ is the standard normalized measure on the unit circle: 
arc length divided by $2\pi$.

We consider continuous homomorphism $\torus^m \to \torus^n$. Let~$\mathrm{Hom}(\torus^m, \torus^n)$ be the collection of such homomorphisms. We are interested in measuring the complexity of elements of~$\mathrm{Hom}(\torus^m, \torus^n)$. Our measure of complexity will be called the Boyd height. 
Note, $\mathrm{Hom}(\torus^m, \torus^n)$ forms an Abelian group under multiplication of functions. In fact, 
the group~$\mathrm{Hom}(\torus^m, \torus^n)$
is canonically isomorphic to the additive group $M_{n, m}(\bZ)$ of $n$ by $m$ matrices $A$ with coefficients in $\bZ$. Now we delve into this connection in more detail. 

Let $A \in M_{n, m}(\bZ)$
be an $n$ by $m$ matrix with integer coefficients. Let $\tilde A \colon \bR^m \to \bR^n$
be the corresponding linear homomorphism described by mapping a column vector $\mb r \in \bR^m$
to the column vector $A \cdot \mb r$. Then we have a commutative diagram
of continuous
homomorphisms (in the category of topological groups):
 $$
 \begin{tikzcd} [column sep = large]
\bR^m  \arrow[r, "\tilde A"]  \arrow[d] & \bR^n \arrow[d]    \\
\bR^m/\bZ^m  \arrow[r, "\bar A"]  \arrow[d, "\mathrm{exp}"] & \bR^n/\bZ^n  \arrow[d, "\mathrm{exp}"]   \\
\torus^m  \arrow[r,  "q_A"]  & \torus^n    \\
\end{tikzcd}
$$
The unlabeled maps above are just the canonical projections.  The map $\overline A$ is the induced
map on the quotient group; it is well-defined since $A$ has integer coefficients.
The map $q_A$ can be seen to be the map describe above in our introduction.
The exponential map $\bR^m / \bZ^m \to T^m$ is the isomorphism 
$$
(z_1, \ldots, z_m) \mapsto \big( \exp(2 \pi z_1 i), \ldots, \exp(2 \pi z_m i) \big),
$$
and the exponential map $\bR^n / \bZ^n \to \torus^n$ is defined similarly.

It turns out that every continuous homomorphism $\torus^m \to \torus^n$ is of the form $q_A$
for some (unique) matrix~$A$ in $M_{n, m}(\bZ)$.
(This is likely well-known, but for the convenience
of the reader we justify it in the appendix).
Further, $A \mapsto q_A$ defines the isomorphism between the (additive) group~$M_{n,m}(\bZ)$ and
the (multiplicative) group~$\mathrm{Hom}(\torus^m, \torus^n)$.

If $\mb r \in \bZ^n$, then we can view $\mb r$ as an element of $M_{n, 1}(\bZ)$
and so $q_{\mb r}$ will denote a homomorphism $\torus \to \torus^n$. These are the types
of homomorphisms present in the classical Boyd--Lawton theorem. (Occasionally, 
$\mb r \in \bZ^n$ will also be thought of as a row vector in~$M_{1, n}(\bZ)$
and so $q_{\mb r}$ will denote a homomorphism [a character] $\torus^n \to \torus$.
Context will make the meaning behind the notation clear.) 

Suppose $A_1 \in M_{n, m}(\bZ)$ and $A_2 \in M_{m, l}(\bZ)$.
Combining commutative diagrams gives us the following commutative diagram:
 $$
 \begin{tikzcd} [column sep = large]
\bR^l \arrow[r, "\tilde A_2"] \arrow[d] 
     \arrow[rr, bend left, "\tilde A_1 \circ \tilde A_2 \, =\;    \widetilde{A_1 A_2}"]  
        & \bR^m  \arrow[r, "\tilde A_1"]  \arrow[d] & \bR^n \arrow[d]    \\
\torus^l  \arrow[r,  "q_{A_2}"] \arrow[rr, swap, bend right, "q_{A_1} \circ\,  q_{A_2}"]  & 
\torus^m  \arrow[r,  "q_{A_1}"]  & \torus^n    \\
\end{tikzcd}
$$
The bottommost arrow must be $q_{A_1 A_2}$ because the diagram commutes.
Thus we have the following:

\begin{proposition} \label{comp_prop}
Let  $A_1 \in M_{n, m}(\bZ)$ and $A_2 \in M_{m, l}(\bZ)$. Then the associated
continuous homomorphisms $q_{A_1} \colon \torus^m \to \torus^n$ and $q_{A_2} \colon \torus^l \to \torus^m$
satisfy the following composition law:
$$
q_{A_1} \circ q_{A_2} = q_{A_1 A_2}.
$$
\end{proposition}

\begin{remark}
We can think of this proposition as giving an equivalence of categories between the category with 
arrows equal to integer matrices (with objects equal to positive integers, say)  and the category of finite dimensional real tori with arrows equal to continuous homomorphisms.
\end{remark}


\section{Boyd Heights of Homomorphisms} \label{height_section}

In~\citep[p.~118]{boyd1981} Boyd defines a measure of the complexity of a nonzero integer vector $\mb r \in \bZ^n$
where $n\ge 2$:
$$
\mu(\mb r) =  \min  \{ \| \mb v \|_\infty \colon \mb v \cdot  \mb r = 0,  \, \mb v \in \bZ^n, \,  \mb v \ne 0 \}.
$$
Since $\mathrm{Hom}(\torus, \torus^n)$ is canonically isomorphic to $\bZ^n \cong M_{n, 1}(\bZ)$, we can 
just as well regard
this measure as giving a measure of complexity to elements $q_{\mb r}$ of $\mathrm{Hom}(\torus, \torus^n)$.
In particular, we will write $\mu\left(q_{\mb r}\right)$ for $\mu(\mb r)$ when we wish to emphasize the
complexity, or ``height'', of the homomorphism $q_{\mb r}$ instead of the vector~$\mb r$. 
In order to make the case that this is a natural measure of complexity on~$\mathrm{Hom}(\torus, \torus^n)$ we can use Proposition~\ref{comp_prop} to rewrite the formula for this height in terms of homomorphisms as follows:
$$
\mu\left(q_{\mb r}\right) =  \min  \{ \| q_{\mb v}  \| \colon q_{\mb v} \circ q_{\mb r}= 0,  \, q_{\mb v} \in \mathrm{Hom}(\torus^n, \torus), \,  q_{\mb v} \ne 0 \}.
$$
where $0$ refers to a zero map and where we define our norm $\| \cdot \|$ on  $\mathrm{Hom}(\torus^n, \torus)$
by the formula $\| q_{\mb v} \| = \| \mb v \|_{\infty}$.
In other words, we measure the complexity of $q_{\mb r}$ by finding the smallest nonzero
homomorphism $q_{\mb v}$ that ``annihilates'' 
$q_{\mb r}$ in the sense of having the image of $q_{\mb r}$
be contained in the kernel of~$q_{\mb v}$.

In honor of Boyd we call this measure of complexity $\mu\left(q_{\mb r}\right)$ the \emph{Boyd height}
of the homomorphism $q_{\mb r}$.
In \cite[page~1409]{Brunault2024} this notion of height is generalized to matrices $A \in M_{n, m}(\bZ)$
which we express as follows:\footnote{The authors of \cite{Brunault2024} write $\rho(A)$
for what we call $\mu(A)$, except we switch the role of rows and columns in the matrix~$A$.
So their $\rho(A)$ is expressed in terms of right multiplication by a column vector~$\mb v$, while
we use left multiplication by a row vector $\mb v$. We do this switch to emphasize the
connection to homomorphisms.
Also, our use of $\mu$ reflect Boyd's original paper.
We use the term ``Boyd height'' to encompass this generalization by~\cite{Brunault2024} although
the designation ``Boyd--Brunault--Guilloux--Mehrabdollahei--Pengo height''  might be more
accurate.}
$$
\mu\left(A\right) 
\; \defeq \;  
\min 
\left\{ \| \mb v\|_\infty  \colon
    \text{$\mb v \in \bZ^n, \mb v \ne {\mb 0}$, and $\mb v \cdot A = \mb 0$} 
 \right\}.
$$
This also can be phrased in terms of homomorphisms which is our preferred formulation:

\begin{definition}\label{boyd_height_definition}
Let the map~$q_{A} \colon \torus^m \to \torus^n$ be the continuous homomorphism associated to a 
matrix~$A \in M_{n, m}(\bZ)$. The \emph{Boyd height} $\mu\left(q_A\right)$ of the homomorphism~$q_A$ is defined as follows:
$$
\mu\left(q_A\right) \; \defeq \;   \min  \left\{ \| q_{\mb v}  \| \colon q_{\mb v} \circ q_{A}= 0,  \, q_{\mb v} \in \mathrm{Hom}(\torus^n, \torus), \,  q_{\mb v} \ne 0 \right\}.
$$
As above, $0$ refers to a zero map, and we define our norm $\| \cdot \|$ on $\mathrm{Hom}(\torus^n, \torus)$
by the formula~$\| q_{\mb v} \| \; \defeq \; \| \mb v \|_{\infty}$ where $\| \mb v \|_{\infty}$
is the absolute value of the largest coordinate of $\mb v$.

If no such $q_{\mb v}$ exists (so the above set is empty), then define $\mu\left(q_A\right)$ to be $\infty$.
\end{definition}

As expected, Boyd heights are unbounded. In fact, we have the following mild generalization of~\citep[p.~118]{boyd1981}:

\begin{proposition} \label{unbounded_lemma}
For any positive integers $m, n$,  the set of values of the Boyd height 
on~$\mathrm{Hom}(\torus^m, \torus^n)$ is unbounded.
In fact, if $b$ is any given positive integer and if 
$$
A = \begin{bmatrix}
1 & 1 & \cdots & 1 \\
b & b & \cdots & b \\
\vdots & \vdots & \cdots & \vdots \\
b^{n-1} & b^{n-1} & \cdots & b^{n-1}
\end{bmatrix},
$$
then
$$
\mu\left(q_{A}\right)  = b.
$$
\end{proposition}

\begin{proof}
If $\mb v = (-b, 1, 0, \ldots, 0)$ then $\mb v \cdot A = 0$ and $\| \mb v \|_\infty = b$. So $\mu\left(q_{A}\right)  \le b$. 

In order to show that~$\mu\left(q_{A}\right)  \ge b$, suppose
  that $\mb v \cdot A = 0$ for some $\mb v$ with $\| \mb v \|_\infty < b$.
 If~$b = 1$ then this means $\mb v = 0$ and we are done; so we can assume $b > 1$.
Since $\|\mb v\|_\infty < b$, we can write~$\mb v = \mb v^+ - \mb v^{-}$ where each coordinate
of  $ \mb v^+$ and $ \mb v^-$ is in the set $\{ 0, 1, \ldots, b-1\}$. 
From~$\mb v \cdot A = 0$ we have  $\mb v^+ \cdot A = \mb v^- \cdot A$. This means $N = \mb v^+ \cdot \mb r = \mb v^- \cdot \mb r$ where $\mb r$ is the column vector with entries $(1, b, b^2, \ldots, b^{n-1})$.
So we have two base $b$ expressions for $N$. 
By uniqueness of base $b$-expansions, $\mb v^+  = \mb v^-$. In other words, $\mb v = \mb v^+ - \mb v^{-} = 0$.
\end{proof}

\begin{remark}
Although our preferred formulation emphasizes homomorphisms, 
we will still sometimes write the Boyd height $\mu\left(q_A\right)$ simply as $\mu(A)$ and 
the vector version $\mu\left( q_{\mb r}\right)$ as~$\mu(\mb r)$ in contexts where the matrices are the focus. 
For example, one might be tempted to 
write~``$\mu(A) = b$''
in the above lemma since it is conceptually a fact about matrices.
\end{remark}

Note we allow $\mu\left( q_A \right)$ to be infinite. This occurs in the following situation:

\begin{proposition} \label{infinite_height_lemma}
Let $q_A \colon \torus^m \to \torus^n$ be a continuous homomorphism where~$A \in M_{n, m}(\bZ)$.
Then the following are equivalent.
\begin{enumerate}
\item
The Boyd height $\mu\left( q_A \right) = \mu(A)$ is infinite.
\item
The matrix $A$ has rank $n$.
\item
The homomorphism $q_A \colon \torus^m \to \torus^n$ is surjective.
\end{enumerate}
In particular, a necessary condition for the height $\mu\left( q_A \right)$ to be infinite is $n \le m$.
\end{proposition}

\begin{proof}
(1) $\implies$ (2). We prove the contrapositive: suppose that $A$ has rank strictly less than $n$. Consider the linear map $\bQ^n \to \bQ^m$
defined by the rule~$\mb v \mapsto \mb v \cdot A$. The rank-nullity theorem implies that the null space has dimension $n - \mathrm{rank} (A) > 0$. Thus there
is a nonzero vector~$\mb v \in \bQ^n$ such that~$\mb v \cdot A = 0$. Replacing $\mb v$ by a suitable positive integer multiple of $\mb v$, we can assume that~$\mb v \in \bZ^n$.
Since $\mb v \cdot A = 0$ we have $\mu(A) \le \| \mb v \|_\infty < \infty$.

(2) $\implies$ (3). If $A$ has rank $n$ then the corresponding linear transformation $\tilde A \colon \bR^m \to \bR^n$ must have image all of $\bR^n$. This forces the induced map $q_A$
to be surjective as well.

(3) $\implies$ (1).  Suppose on the contrary that $\mu(q_A)$ is finite. This means that there is a nonzero $q_{\mb v} \in\mathrm{Hom}(\torus^n, \torus)$ such that $q_{\mb v} \circ q_A$ is the zero map.
But  $q_A$ is surjective, so $q_{\mb v}$ must be the zero map, a contradiction.
\end{proof}

We can enforce lower bounds for the Boyd-height 
of a composition:

\begin{proposition}\label{TP_prop}
Let $q_{A} \colon \torus^n \to \torus^m$ and $q_{B} \colon \torus^m \to \torus^l$ be continuous homomorphisms and
let~$b \in \bR$ be a given bound.
If \ 
$\mu \left( q_{A} \right) \ge l\|B\|_\infty b$ \ 
and  \ 
$\mu \left( q_{B} \right) \ge  b$ 
\ 
then  
$$
\mu  \left( q_{B} \circ q_{ A} \right)  \ge b.
$$
\end{proposition}

\begin{proof}
We work with the underlying matrices:
we assume $A\in M_{m,n}(\ZZ)$ and $B\in M_{l,m}(\ZZ)$ are such that $\mu(A)\geq l\|B\|_\infty b$ and $\mu(B)\geq  b$ where $b$ is the given bound,
and we want to show that $\mu(BA) \ge b$.
It suffices to show that 
if $\mathbf{v}\in\ZZ^l$ 
is a nonzero vector such that~$\mathbf{v}\cdot BA=\mathbf{0}$ then $\| \mb v \|_\infty \ge b$.
(If no such nonzero vector $\bfv$ exists, then $\mu(BA)=\infty \ge b$, and we are done).

We have two cases depending on $\bfv \cdot B$.

\noindent \underline{Case 1:} $\mathbf{v}\cdot B\neq \mathbf{0}$.  Let $\mathbf{v}\cdot B=\mathbf{w}=(w_i)_{1\leq i \leq m}$. 
Here $\mb w \cdot A = \mb 0$ so 
$\|\mathbf{w}\|_\infty\geq \mu(A)\geq l\|B\|_\infty b$.
Let $j$ be such that $|w_j|=\|\mathbf{w}\|_\infty$, and let $\mathbf{b}_j = (b_{ij})$ be the $j$th column of $B$.  Then 
\begin{align*}
l\|B\|_\infty b \leq  |w_j| = |\mathbf{v}\cdot \mathbf{b}_j| = \left |\sum_{i=1}^l v_ib_{ij}\right |\leq  \sum_{i=1}^l |v_ib_{ij}| \leq  \sum_{i=1}^l|v_i|\|B\|_\infty  \leq&  \sum_{i=1}^l\|\mathbf{v}\|_\infty\|B\|_\infty\\ 
 = &l\|\mathbf{v}\|_\infty\|B\|_\infty.
\end{align*}
Thus, $l\|B\|_\infty b\leq l\|\mathbf{v}\|_\infty\|B\|_\infty $ and therefore $\|\mathbf{v}\|_\infty\geq b.$

\medskip\noindent \underline{Case 2:} 
$\mathbf{v}\cdot B=\mathbf{0}.$  In this case, since $\mathbf{v}$ is nonzero, we must have 
\[\|\mathbf{v}\|_\infty\geq\mu(B)\geq b.\]

\medskip\noindent
In either case, $\|\mathbf{v}\|_\infty\geq b$
as desired. 
\end{proof}


\section{Boyd--Lawton Collections} \label{section_BL_Collections}

The purpose of this section is to verify that the four collections of functions discussed 
in the introduction are indeed Boyd--Lawton collections (Definition~\ref{BL collection def}) and so are subject to our main theorem (Theorem~\ref{main_theorem} above). 

Before we do this, we need as a tool a basic result about nonvanishing of Laurent polynomials under substitution. For this we need to extend the notation $P^{(\mb r)}$ for power
substitution to matrices $A$.
Suppose that~$P \in \bC\left[Z^{\pm 1}_1, \ldots, Z^{\pm 1}_n\right]$ is a Laurent polynomial of $n$ variables and $A\in M_{n,m}(\bZ)$ is a matrix.
We form the Laurent polynomial $P^{(A)}$ of $m$ variables:
$$
P^{(A)}(Z_1, \ldots, Z_m) \; \defeq \; 
P\left( Z_1^{a_{11}}  Z_2^{a_{12}} \cdots  Z_m^{a_{1m}},  \;  \ldots \; , \; Z_1^{a_{n1}}  Z_2^{a_{n2}} \cdots  Z_m^{a_{nm}} 
\right).
$$
This definition was chosen to be compatible with the composition $P \circ q_A$.
More precisely, although $P$ and $P^{(A)}$ are technically algebraic objects, namely Laurent polynomials, we can slightly abuse notation and think of them as functions on $\torus^n$ and $\torus^m$ respectively. In this case we see that
$$
P^{(A)} = P \circ q_A
$$
as a function on  the torus $\torus^m$.\footnote{Although $P^{(A)}$ and $P \circ q_A$ are interchangeable when regarded as functions on $\torus^m$,  we still find it important to keep both notations.
As a purely algebraic Laurent polynomial, the functional composition notation $P \circ q_A$ is misleading, so we will use $P^{(A)}$ in algebraic settings. For example, Mahler measures and their extensions are
viewed as invariants of Laurent polynomials themselves and not of functions on tori, so we like to use~$P^{(A)}$ in the context of such Mahler measures. We prefer $P \circ q_A$ in analytic settings such as in integrals on~$\torus^m$.

Finally, we note that other authors use $P_{\mb r}$ and $P_A$ for what we call  $P^{(\mb r)}$ and $P^{(A)}$. We like this superscript notation, with parenthesis, since it is suggestive of the role of the coordinates
of $\mb r$ and $A$ as exponents, and the parentheses reminds us that these are applied as a substitution inside the Laurent polynomial $P$. Also superscripts are good when you have Laurent polynomials $P_i$ denoted with subscripts already in use; we avoid awkward expressions such as $P_{i \mb r}$, using the more pleasant  $P^{(\mb r)}_i$ instead.}

\begin{lemma} \label{Pqr_nonvanishing}
Let $P \in \bC\left[Z^{\pm 1}_1, \ldots, Z^{\pm 1}_n\right]$ be a nonzero Laurent polynomial. Then
there is bound~$M$ such that $P^{(A)}$ is not zero if $A$ is an integer matrix with $n$ rows such that~$\mu(A) \ge M$. 
\end{lemma}

\begin{proof}
Implicit in the classical Boyd--Lawton theorem is the assertion that
there is a bound $M$ such that 
$\log | P| \circ q_{\mb t} = \log| P \circ q_{\mb t} |$ defines an integrable function on~$\torus$
for any $\mb t \in \bZ^n = M_{n,1}(\bZ)$
such that~$\mu({\mb t}) \ge M$. 
In particular, the function~$P \circ q_{\mb t}$
cannot be identically zero on $\torus$ when~$\mu({\mb t}) \ge M$.

Suppose $A \in M_{n, m}(\bZ)$ is such that $\mu(A) \ge M$.
Choose $\mb r \in \bZ^m$ so that $\mu\left(\mb r\right) \ge n \|A\|_\infty M$, and
let $q_{\mb t}$ be the composition $q_A \circ q_{\mb r}$.
Such an $\mb r$ exists by Proposition~\ref{unbounded_lemma}.
By Proposition~\ref{TP_prop}, we have $\mu(q_{\mb t}) \ge M$, meaning
that the composition~$P \circ q_{\mb t} = P \circ q_A \circ q_{\mb r}$ is not
identically zero on~$\torus$. This in turn prevents $P \circ q_A$ from being
identically zero on~$\torus^m$. 
As functions on~$\torus^m$, we have $P \circ q_A = P^{(A)}$, so we conclude
that $P^{(A)}$ is a nonzero Laurent polynomial.
\end{proof}

\subsection{The Case of Continuous Functions}

Our first Boyd--Lawton collection is the collection of continuous functions on real tori. Let $\mc C_1 (\torus^n)$ be the collection of continuous functions~$\torus^n \to \bC$.

\begin{proposition} \label{prop: C_1}
The set  $\mc C_1 = \bigcup_{n \ge 1} \mc C_1 (\torus^n)$ is a Boyd--Lawton collection
\end{proposition}

\begin{proof}
Continuous functions are integrable on $\torus^n$ since $\torus^n$ is compact and $\tau_n$ is the Haar measure, so part~(1) of Definition~\ref{BL collection def} is satisfied.

For a continuous $g \colon \torus^n \to \bC$ with $n \ge 2$ we have
$$
\int_{\torus^n}  g \, d\tau_n= \lim_{\mu(\mb r)\to \infty}\;  \int_\torus g \circ q_{\mb r} \, d \tau
$$
by~\cite[Lemma~1]{boyd1981}. 
So part~(2) of Definition~\ref{BL collection def} is satisfied.

Finally, since the composition of continuous functions is continuous, 
part~(3) of Definition~\ref{BL collection def} is satisfied.
\end{proof}


\subsection{The Case of Classical Mahler Measures}

Our second Boyd--Lawton collection is the usual collection of functions on $\torus^n$ used to define Mahler measures of Laurent polynomials.
Let $\mc C_2 (\torus^n)$ be the collection of functions of the form
$$
(z_1, \ldots, z_n) \mapsto \log | P(z_1, \ldots, z_n) |
$$
where $P \in \bC\left[Z^{\pm 1}_1, \ldots, Z^{\pm 1}_n\right]$ is a nonzero Laurent polynomial.

\begin{proposition}\label{prop: C_2}
The set  $\mc C_2 = \bigcup_{n \ge 1} \mc C_2 (\torus^n)$ is a Boyd--Lawton collection
\end{proposition}

\begin{proof}
Throughout this proof let $g$ be an arbitrary element of $\mc C_2$, so $g = \log |P|$ where~$P \in \bC\left[Z^{\pm 1}_1, \ldots, Z^{\pm 1}_n\right]$ is a nonzero Laurent polynomial. 

The integral
$$
\int_{\torus^n}  g \, d\tau_n = \int_{\torus^n}  \log |P| \, d\tau_n 
$$
is the Mahler measure of $P$. This is well-defined by a basic result on Mahler measures (see~\cite[Proposition 3.1]{Brunault_Zudilin_2020}). 
Thus part~(1) of Definition~\ref{BL collection def} is satisfied.

If $n\ge 2$, then the classical Boyd--Lawton theorem gives
$$
\int_{\torus^n}  g \, d\tau_n= \lim_{\mu(\mb r)\to \infty}\;  \int_\torus g \circ q_{\mb r} \, d \tau.
$$
Thus, part~(2) of Definition~\ref{BL collection def} is satisfied.

Finally, we appeal to Lemma~\ref{Pqr_nonvanishing} to get a bound $M$ such that $P^{(A)}$ is a nonzero Laurent polynomial for all
integer matrices $A$ with $n$ rows such that $\mu(A) \ge M$. Furthermore, as functions on $\torus^m$,
$$
\log \left| P^{(A)} \right| = \log \left | P \circ q_A \right| = \log |P| \circ q_A = g \circ q_A.
$$
Thus $g \circ q_A$ is in $\mc C_2(\torus^m)$ when $q_A \colon \torus^m \to \torus^n$ is such that $\mu\left( q_A \right) \ge M$. Part~(3) of Definition~\ref{BL collection def} is now established.
\end{proof}


\subsection{The Case of More General Mahler Measures}

Next, we consider functions associated with generalizations of Mahler measures.

In our third Boyd--Lawton collection $\mc C_3$, the functions $\mc C_3(\torus^n)$
are those of the form
$$
(z_1, \ldots, z_n) \mapsto  \max \left\{ 
    \log | P_1(z_1, \ldots, z_n) |, \ldots, \log| P_k(z_1, \ldots, z_n)| 
    \right\},
$$
where each $P_i$ is a nonzero  Laurent polynomial in the ring  $\bC[Z^{\pm 1}_1, \ldots, Z^{\pm 1}_n]$.
These are the functions related to \emph{generalized Mahler measures} introduced
in~\cite{GonOyanagi2004}.

In our fourth and final Boyd--Lawton collection $\mc C_4$, the functions $\mc C_4(\torus^n)$
are those of the form
$$
(z_1, \ldots, z_n) \mapsto  
    \log | P_1(z_1, \ldots, z_n) | \cdot \log | P_2(z_1, \ldots, z_n) | \cdots \log| P_k(z_1, \ldots, z_n)| ,
$$
where each $P_i$ is a nonzero  Laurent polynomial in~$\bC[Z^{\pm 1}_1, \ldots, Z^{\pm 1}_n]$.
These are the functions related to \emph{multiple higher Mahler measures} of~\cite{Kurokawa2008}.

\begin{proposition} \label{prop: C_3}
The sets $\mc C_3 = \bigcup \mc C_3 (\torus^n)$ and $\mc C_4 = \bigcup \mc C_4 (\torus^n)$ are Boyd--Lawton collections.
\end{proposition}

\begin{proof}
If $g$ is in $\mc C_3 (\torus^n)$ or $\mc C_4 (\torus^n)$, then 
$
\int_{\torus^n}  g \, d\tau_n 
$
is $m_{n, \max}(P_1, \ldots, P_k)$ or $m_{n, \max}(P_1, \ldots, P_k)$ for some set of 
Laurent polynomials in $\bC[Z^{\pm 1}_1, \ldots, Z^{\pm 1}_n]$. By Lemma~\ref{lem: PolyLaurSwap1} below (as illustrated in the proof of Lemma~\ref{lem: PolyLaurSwap2}), we can assume that each $P_i$
is actually a polynomial. The convergence of $\int_{\torus^n}  g \, d\tau_n$  now follows from Theorem~4.1 of \cite{Issa2013}.
Therefore,  part~(1) of Definition~\ref{BL collection def} is satisfied.  

 Part~(2) of Definition~\ref{BL collection def} essentially follows from Theorem~1.2 of \cite{Issa2013}. However, this theorem of \cite{Issa2013}
 needs to be strengthened a bit, which we do below (Theorem~\ref{extendedIssa}). 
With this part~(2) of Definition~\ref{BL collection def} is satisfied.

The argument for part~(3) of Definition~\ref{BL collection def} is similar to the argument in Proposition~\ref{prop: C_2}.
If $g$ is in $\mc C_3 (\torus^n)$ or $\mc C_4 (\torus^n)$ then $g$ is built out of a set $P_1, \ldots, P_k$ of Laurent polynomials. 
We  we appeal to Lemma~\ref{Pqr_nonvanishing} to get a bound $M_i$ such that $P_i^{(A)}$ is a nonzero Laurent polynomial for all
integer matrices $A$ with $n$ rows such that $\mu(A) \ge M_i$.  If we simply take $M$ to be the maximum of the $M_i$ then
$P_1^{(A)}, \ldots, P_1^{(A)}$ are nonzero Laurent polynomials when $\mu(A) \ge M$.
Since $g \circ q_A$ is built out of $P_1^{(A)}, \ldots, P_1^{(A)}$, this composition will be in~$\mc C_3 (\torus^m)$ or $\mc C_4 (\torus^m)$ if $A$ is an $n$ by $m$ integer matrix with $\mu(A) \ge M$.
For example, if~$g \in \mc C_3(\torus^n)$ is given by
$$
g =   \max \left\{ 
    \log | P_1 |, \ldots, \log| P_k| 
    \right\},
$$
then for $A$ with $\mu(A) \ge M$ we can form the following function in $\mc C_3(\torus^m)$:
\begin{equation}\label{Eq. IdsUsed}
    \begin{split}
        \max \left\{ 
    \log \left| P^{(A)}_1 \right|, \ldots, \log\left| P^{(A)}_k \vphantom{P^{(A)}_1}\right| 
    \right\} 
    &= \max \left\{ 
    \log \left| P_1 \circ q_A \right|, \ldots, \log\left|  P_k\circ q_A\right|      \right\} \\
       &=
        \max \left\{  \log \left| P_1  \right|, \ldots, \log\left|  P_k\right|      \right\}  \; \circ \; q_A \\
        &=
        g \circ q_A.
    \end{split}
\end{equation}
We conclude that  part~(3) of Definition~\ref{BL collection def} holds for $\mc C_3$, and similarly for $\mc C_4$.
\end{proof}

In what follows, we will freely use the identities in Equation \eqref{Eq. IdsUsed}.


\subsection{Extending Issa and Lal\'in's version of Boyd--Lawton} \label{extending_Issa}

In order to complete the proof of Proposition~\ref{prop: C_3}, we need to strengthen Issa and Lal\'in's version of Boyd--Lawton: 

\begin{theorem} [Theorem~1.2 of \cite{Issa2013} and Theorem 30 of \cite{LalinSinha}] \label{IssaThm}
Let $P_1, \dots, P_k \in \bC [Z_1, \ldots, Z_n]$ be nonzero polynomials where $n \ge 2$. Then
    \[\lim_{\mu\left( {\mb r} \right)  \to \infty} m_{n, \max}\left( P_1^{(\mb r)}, \ldots,  P_k^{(\mb r)}\right)
    =m_{n, \max}( P_1, \ldots,  P_k)
    \]
    and
    \[
    \lim_{\mu\left( {\mb r} \right)  \to \infty} m_{n, \operatorname{prod}}\left( P_1^{(\mb r)}, \ldots,  P_k^{(\mb r)}\right)
    =m_{n, \operatorname{prod}}(P_1, \ldots, P_k),
    \]
    where in the limit the $\mb r = (r_1, \ldots, r_n)$ vary only over lattice points in $\bZ^n_{>0}$.
    
        In other words, if $g$ is either
    $$
    \max \{ \log|P_1|, \ldots, \log|P_k| \}
    \qquad\text{or}\qquad
    \ \log|P_1| \cdots \log|P_k| ,
    $$
    then
    $$
\int_{\torus^n}  g \, d\tau_n= \lim_{\mu(q_\mb r)\to \infty}\;  \int_\torus g \circ q_{\mb r} \, d \tau,
$$
 where in the limit the $\mb r = (r_1, \ldots, r_n)$ vary only over lattice points in $\bZ^n_{>0}$.
\end{theorem}

Our goal (Theorem~\ref{extendedIssa} below) is to extend the above result in two directions: 
\begin{enumerate}
    \item We want the result to apply to sequences of Laurent polynomials, not just sequences of ordinary polynomials.

    \item We want to drop the restriction that the lattice points $\mb r \in \bZ^n$ have positive coordinates.
\end{enumerate}

Obtaining goal (1) by itself is almost immediate based on the following observation:

\begin{lemma}\label{lem: PolyLaurSwap1}
    Let $P \in\CC\left[Z_1^{\pm1},\dots, Z_n^{\pm1}\right]$ be a nonzero Laurent polynomial, then there is a polynomial $Q \in \bC [Z_1, \ldots, Z_n]$
    such that
    $$
    | P(z_1, \dots, z_n) | = |Q (z_1, \dots, z_n)|
    $$
    for all $(z_1, \ldots, z_n) \in \torus^n$.
\end{lemma}

\begin{proof}
    Write $P(Z_1,\dots, Z_n)=Z_1^{-k_1}\cdots Z_n^{-k_n} \cdot Q(Z_1,\dots, Z_n)$ where the~$k_i\in \ZZ_{\geq 0}$ and $Q$ is a polynomial in $\CC[Z_1,\dots, Z_n]$. For $(z_1,\dots, z_n)\in \torus^n$ we have
    $|z_1| = \ldots = |z_n| = 1$ so
    \begin{align*}
            \log\bigl|P(z_1,\dots, z_n)\bigr|      &=\log\big|z_1^{-k_1}\cdots z_n^{-k_n} \cdot Q (z_1,\dots, z_n)\big|\\
            &=\log\big|z_1^{-k_1}\cdots z_n^{-k_n}\big|+ \log\big|Q(z_1,\dots, z_n)\big|\\
            &=\sum_{j=1}^n-k_j\log|z_j| \ +\log\big|Q(z_1,\dots, z_n)\big|\\
            &=\log\big|Q(z_1,\dots, z_n)\big|. \qedhere
    \end{align*}
\end{proof}

From this we can extend \cite[Theorem~1.2]{Issa2013} to Laurent polynomials:

\begin{lemma}\label{lem: PolyLaurSwap2}
   Suppose $P_1, \ldots, P_k \in\CC\left[Z_1^{\pm1},\dots, Z_n^{\pm1}\right]$ are  nonzero Laurent polynomials where~$n \ge 2$.
If $g$ is either
    $$
    \max \{ \log|P_1|, \ldots, \log|P_k| \}
    \qquad\text{or}\qquad
    \ \log|P_1| \cdots \log|P_k| ,
    $$
    then
    $$
\int_{\torus^n}  g \, d\tau_n= \lim_{\mu(q_\mb r)\to \infty}\;  \int_\torus g \circ q_{\mb r} \, d \tau,
$$
 where in the limit  the $\mb r = (r_1, \ldots, r_n)$ vary only over lattice points in $\bZ^n_{>0}$.
\end{lemma}

\begin{proof}
We use Lemma~\ref{lem: PolyLaurSwap1} to obtain polynomials $Q_1, \ldots, Q_k  \in \bC[Z_1, \ldots, Z_n]$ such that 
    $$
    | P_i(z_1, \dots, z_n) | = |Q_i (z_1, \dots, z_n)|
    $$
    for all $(z_1, \ldots, z_n) \in \torus^n$.
In particular,  $g$ is either
    $$
    \max \{ \log|Q_1|, \ldots, \log|Q_k| \}
    \qquad\text{or}\qquad
    \ \log|Q_1| \cdots \log|Q_k| .
    $$
To finish the argument, we apply Theorem~\ref{IssaThm} (Theorem 1.2 of \cite{Issa2013}).
\end{proof}

The challenge now is to remove the positivity condition and allow $\mb r$ to be any $\mb r \in \bZ^n$ in the limit. 
\emph{A priori}, we do not assume all the components of $\mb r$ are nonzero, but we can assume~$\mu(\mb r)$ is large
which forces the components to be nonzero:

\begin{lemma}\label{one_lemma}
Suppose $\mb r \in \bZ^n$. If any component of $\mb r$ is $0$, then $\mu(\mb r) = 1$. In other words,~$\mu(\mb r) >1$
implies that the components of $\mb r$ are nonzero.
\end{lemma}

\begin{proof}
For example, suppose the first component of $\mb r$ is $0$. Then $\mb v = (1, 0, \ldots, 0)$ is orthogonal to $\mb r$, and $\| \mb  v\|_\infty = 1$.
This forces $\mu(\mb r)$ to be~$1$.
\end{proof}

Our idea is to replace $\mb r = (r_1, \ldots, r_n)$ whenever possible with $\mb r^+ = \left( |r_1|, \ldots, |r_n| \right)$, which does not change the Boyd height:

\begin{lemma} \label{same_mu}
Suppose $\mb r = (r_1, \ldots, r_n) \in \bZ^n$ and $\mb r^+ = \left( |r_1|, \ldots, |r_n| \right)$. Then 
$$
\mu \left( \mb r \right)
=
\mu \left( \mb r^+ \right).
$$
\end{lemma}

\begin{proof}
If $\mb v$ is orthogonal to $\mb r$, then we simply change the signs of the $i$ coordinate of $\mb v$ for every $i$ with $r_i < 0$. This results in a vector $\mb v'$ orthogonal to $\mb r^+$ with $\|\mb v\|_\infty= \|\mb v'\|_\infty$.
In fact, this process defines a norm preserving bijection between the set of nonzero vectors in~$\bZ^n$ orthogonal to $\mb r$ and the set of nonzero vectors in~$\bZ^n$ orthogonal to~$\mb r^+$.
The result now follows from the definition of the Boyd height $\mu$.
\end{proof}

Our key trick is to factor any $\mb r \in \bZ$ as $D \cdot \mb r^+$ where $D$ is a diagonal matrix with diagonal entries $\pm 1$ (and indeed $D$ describes the isomorphism of the above proof).
By Proposition~\ref{comp_prop}, this means that $g \circ q_{\mb r}$ is equal to $g \circ q_D \circ  q_{\mb r^+}$.
This suggests we use Lemma~\ref{lem: PolyLaurSwap2} on functions of the form $g \circ q_D$. The good news is that this does not change the integral for $g$:


\begin{lemma}\label{invariance_lemma2}
Let $A$ be an invertible $n$ by $n$ matrix with integer coefficients and let $g$ be an integrable complex valued function on $\torus^n$.
Then
$$
 \int_{\torus^n} g\circ q_A \ d \tau_n =  \int_{\torus^n} g \ d \tau_n.
$$
\end{lemma}

\begin{proof}
Recall the commutative diagram (using notation established above; for example $\tilde A$ is the linear transformation associated with $A$):
 $$
 \begin{tikzcd} [column sep = large]
\bR^n  \arrow[r, "\tilde A"]  \arrow[d] & \bR^n \arrow[d]    \\
\bR^n/\bZ^n  \arrow[r, "\bar A"]  \arrow[d, "\mathrm{exp}"] & \bR^n/\bZ^n  \arrow[d, "\mathrm{exp}"]   \\
\torus^n  \arrow[r,  "q_A"]  & \torus^n    
\end{tikzcd}
$$
The degree of 
the induced maps $\overline A$ and $q_A$ are $d = | \det A|$; in other words, these maps are~$d$-to-$1$ (covering) maps.
In particular, if $I = [0, 1]$, then the image $\tilde A (I^n)$ consists of~$d = | \det A|$ copies of fundamental domain for $\bZ^n$.\footnote{In what follows we only actually need the case when $A$ is a diagonal matrix. In this case the decomposition of $\tilde A (I^n)$ is easy to picture: it is  an  $n$-box with edge length given by the absolute values of the diagonal
entries of~$A$, and the total measure of $\tilde A (I^n)$ is equal to $|\det A|$. In this case, $\tilde A (I^n)$ decomposes into $|\det A|$ distinct~$\bZ^n$ translations of $I^n$, and these translations
intersect in measure zero.}

We use the above diagram to work with integrals on~$I^n$ (and translations). To do so we need to integrate
the periodic function
$$
\overline g(z_1, \ldots, z_n) \; \defeq \; g\big( \exp (2 \pi z_1 i), \ldots, \exp (2 \pi z_n i) \big)
$$
on subsets of $\bR^n$ such as $I^n$.
Note that $g$ is measurable, so $\overline g$ is measurable, and $g$ is integrable, so $\overline g$ is integrable on $I^n$. In fact, it follows that $\overline g$ is integrable on
any bounded subset of $\bR^n$, including $\tilde A(I^n)$. This fact allows us to use the change of variables formula (See Theorem~2.9, Chapter~VIII, of~\cite[page 191]{lang_dman}).
We make use of the substitution $\mb u = \tilde A (\mb x)$ with Jacobian given by~$|\det A|$. 
We have
\begin{eqnarray*}
\int_{\tilde A (I^n)} \overline g\left ( \mb u\right)\ d \mb u
&=&
\int_{I^n} \overline g\left (\tilde A (\mb z)\right)\ |\det A|  \, d \mb z \\
&=&
|\det A| \int_{\torus^n} g \circ q_A \ d\tau_n. \\
\end{eqnarray*}

The $\bZ^n$-periodicity of $\overline g$ together with a decomposition of $\tilde A(I^n)$ into $d = |\det A|$ fundamental domains $F_1, \ldots, F_d$ of $\bZ^n$ implies that
\begin{eqnarray*}
\int_{\tilde A (I^n)} \overline g\left ( \mb u\right)\ d \mb u
&=&
\sum_{i=1}^{d} \int_{F_i} \overline g (\mb u) \ d\mb u\\
&=&
\sum_{i=1}^{|\det A|} \int_{\torus^n} g \ d\tau_n\\
&=&
|\det A| \int_{\torus^n} g  \ d\tau_n. \\
\end{eqnarray*}
Combining the two equations for $\int_{\tilde A (I^n)} \overline g\left ( \mb u\right)\ d \mb u$ yields
\[\int_{\torus^n} g \circ q_A \ d\tau_n  =  \int_{\torus^n} g  \ d\tau_n. \qedhere \]
\end{proof}

Now we have all the tools necessary to extend Theorem~\ref{IssaThm}  (Theorem 1.2 of \cite{Issa2013}) to our more general setting (see Corollary~\ref{extendedIssa2} for an even more general extension).

\begin{theorem} \label{extendedIssa}
Suppose $P_1, \ldots, P_k \in\CC\left[Z_1^{\pm1},\dots, Z_n^{\pm1}\right]$ are  nonzero Laurent polynomials with $n\ge 2$.
If $g$ is either
    $$
    \max \{ \log|P_1|, \ldots, \log|P_k| \}
    \qquad\text{or}\qquad
    \ \log|P_1| \cdots \log|P_k| ,
    $$
    then
    $$
\int_{\torus^n}  g \, d\tau_n= \lim_{\mu(q_\mb r)\to \infty}\;  \int_\torus g \circ q_{\mb r} \, d \tau,
$$
 where in the limit the $\mb r = (r_1, \ldots, r_n)$ vary over all lattice points in $\bZ^n$. 
 In other words,
    \[\lim_{\mu\left( {\mb r} \right)  \to \infty} m_{n, \max}\left( P_1^{(\mb r)}, \ldots,  P_k^{(\mb r)}\right)
    =
    m_{n, \max}( P_1, \ldots,  P_k)
    \]
    and
    \[
    \lim_{\mu\left( {\mb r} \right)  \to \infty} m_{n, \operatorname{prod}}\left( P_1^{(\mb r)}, \ldots,  P_k^{(\mb r)}\right)
    =m_{n, \operatorname{prod}}(P_1, \ldots, P_k),
    \]
    where in the limit the $\mb r = (r_1, \ldots, r_n)$ vary over all lattice points in $\bZ^n$.

\end{theorem}

\begin{proof}
    As noted above, our trick is to  factor each $\bfr\in \ZZ^n$ as $D\cdot \bfr^+$ where $D$ is an $n\times n$ diagonal matrix with diagonal entries $\pm 1$ and where $\bfr^+\in\ZZ^n_{\ge 0}$. 
    Let $\mathcal{D}(\pm)$ denote the set of all $2^n$ diagonal~$n\times n$ matrices with diagonal entries $\pm 1$.
    
    Write~$g\mathcal{D}(\pm)$ for the set~$\{g\circ q_D: D\in \mathcal{D}(\pm)\}$.
    By Lemma~\ref{invariance_lemma2} any two elements of $g\mathcal{D}(\pm)$ have the same integral over $\torus^n$.

    Our goal is to show
    $$\lim_{\mu \left( q_{\mb r} \right)  \to \infty} \int_{\torus}  g \circ q_\bfr \, d\mu = \int_{\torus^n} g \, d\mu_n$$ 
    where $\bfr$ varies in $\ZZ^n$. 
To do so, fix $\epsilon>0$.  
We apply Lemma \ref{lem: PolyLaurSwap2} to  $g \circ q_D$ for each of the $2^n$ matrices $D \in \mathcal{D}(\pm)$. 
We get bounds $N_D$ such that whenever $\bfr^+\in\ZZ^n_{>0}$ with $\mu\left(q_{\bfr^+} \right)\geq N_D$, we have
    \[
    \Bigg|
    \int_{\torus} (g\circ q_D) \circ q_{\bfr^+} \ d\mu \, - \int_{\torus^n} g\circ q_D \ d \mu_n \Bigg|<\epsilon.
    \]
   By taking $N$ to be the maximum of the $N_D$, we obtain a uniform bound that works over $g\mathcal{D}(\pm)$.
   Of course we can assume $N > 1$.

    Now, let $\bfr\in\ZZ^n$ with $\mu\left( q_{\mb r} \right) \geq N$. Write $\bfr=D\cdot \bfr^+$ for some $D\in \mathcal{D}(\pm)$. 
    Since $N > 1$, all the components of $\mb r$ are nonzero (Lemma~\ref{one_lemma}), and so $\mb r^+ \in \bZ^n_{>0}$.
Note that $q_{\mb r} = q_D \circ q_{\mb r^+}$ (Proposition~\ref{comp_prop}), so 
     \[
     \int_{\torus} g\circ q_\bfr \ d\mu= \int_{\torus} g\circ q_D \circ q_{\bfr^+} \ d\mu.
     \]
Also $\mu\left( q_{\mb r}^+ \right) = \mu\left( q_{\mb r} \right) \ge N$ (Lemma~\ref{same_mu}), so
    \[
    \left|
    \int_{\torus} g\circ q_D \circ q_{\bfr^+} \ d\mu \, - \int_{\torus^n} g\circ q_D \ d \mu_n 
    \right|<\epsilon.
    \]
    As above $\int_{\torus^n} g\circ q_D \ d\mu_n=\int_{\torus^n} g \  d\mu_n$ (Lemma~\ref{invariance_lemma2}), so
    \[ \left|
           \int_{\torus} g\circ q_\bfr  \ d\mu  \, -   \int_{\torus^n} g  \ d \mu_n 
           \right|<\epsilon. \qedhere
     \]
\end{proof}


\section{Proof of Main Result}

We are ready to prove Theorem~\ref{main_theorem}. First we restate this theorem in a form that aligns with our
proof, and at the same time spells out the limit condition of Theorem~\ref{main_theorem}:

\begin{theorem} \label{main_theorem_alt}
Let $\mc C$ be a Boyd--Lawton collection and let $g \in \mc C(\torus^n)$.
Then for any~$\varepsilon > 0$, there is a real bound $M$ such that if $m$ is a positive integer
and if $q_A \colon \torus^m \to \torus^n$ is a continuous homomorphism with $\mu\left(q_A\right) \ge M$,
then 
$$
\left|   \int_{\torus^{n}} g \, d \tau_{n}
 \; - \; 
\int_{\torus^m} g \circ q_A \, d \tau_m \,
\right|
< \varepsilon.
$$
\end{theorem}

We start by proving the case where $n>1$. We will then use this to prove a stronger version when $n=1$.

\begin{proof}[Proof for $n>1$.]
We begin with the main case where $n>1$, and then use the result in the main case to justify the~$n=1$ case afterwards.
The following diagram illustrates the proof strategy for the main case:
$$
\begin{tikzcd}
\int_{\torus^n} g \; d \tau_n
  \arrow[rr, dash, dashed, "\varepsilon"]  
    \arrow[rd, dash, swap, "\varepsilon/2"] 
&&
\int_{\torus^m} g \circ q_A  \; d \tau_m
 \\
 &
\int_{\torus} g \circ q_A \circ q_{\mb r} \; d \tau
\arrow[ru, dash, swap, "\varepsilon/2"]
 \end{tikzcd}
$$

Fix $g$ and $\varepsilon > 0$ as in the statement of the theorem where $n>1$.
Since $g$ is in a Boyd--Lawton collection we can appeal to the 
single-variable Boyd--Lawton property  for $g$ to obtain a bound~$M_1$ such 
that~$g \circ q_{\mb t}$ is integrable and
\[\left|   \int_{\torus^{n}} g \, d \tau_{n}
 \; - \; 
\int_{\torus} g\circ q_{\mb t} \, d \tau \,
\right|
< \varepsilon/2\]
for all continuous homomorphisms $q_{\mb t} \colon \torus\to \torus^n$ with $\mu\left( q_{\mb t} \right) \ge M_1$.
Again, since $\mc C$ is a Boyd--Lawton collection, we can 
find a bound $M_2$ so that 
$g \circ q_A$ is
integrable and satisfies a single-variable Boyd--Lawton condition whenever~$\mu\left( q_{A} \right) \ge M_2$.

We claim that if $M$ is the maximum of $M_1$ and $M_2$, then this choice of $M$ satisfies the condition of the theorem. 
To justify the claim, fix a continuous homomorphism~$q_A \colon \torus^m \to \torus^n$ such that $\mu\left( q_{A} \right) \ge M$.
Our goal is now to verify
$$
\left|   \int_{\torus^{n}} g \, d \tau_{n}
 \; - \; 
\int_{\torus^m} g \circ q_A \, d \tau_m \,
\right|
< \varepsilon.
$$
for this fixed (but arbitrary) $q_A$.

Since $\mu\left( q_{A} \right) \ge M \ge M_2$, we know that $g \circ q_A$ 
is integrable on $\torus^m$ and that $g \circ q_A$ itself satisfies a Boyd--Lawton condition.
Thus there is a bound $M'$ such that for every~$q_\mb r  \colon \torus \to \torus^m$ satisfying~$\mu\left( q_\mb r \right) \ge M'$, the composition $(g \circ q_A) \circ q_{\mb r} = g \circ q_A \circ q_{\mb r}$ is integrable on $\torus$ and
\[\left|   \int_{\torus^{m}}  g \circ q_A  \, d \tau_{m}
 \; - \; 
\int_{\torus} g \circ q_A \circ q_{\mb r} \, d \tau \,
\right|
< \varepsilon/2.\]
Fix such a bound $M'$, choosing it larger than  $n \| A \|_\infty M$.

By Proposition~\ref{unbounded_lemma}, we may fix a continuous homomorphism $q_\mb r \colon T\to T^m$ such that~$\mu\left( q_\mb r \right) \ge M'$. 
We define $q_{\mb t} = q_A \circ q_{\mb r}$ to be the composition of our two fixed homomorphisms.
By Proposition~\ref{TP_prop}, we have $\mu\left( q_{\mb t}\right) \ge M \ge M_1$. So
\[
\left|   \int_{\torus^{n}}  g   \, d \tau_{n}
 \; - \; 
\int_{\torus} g \circ q_A \circ q_{\mb r} \, d \tau \,
\right|
=
\left|   \int_{\torus^{n}} g \, d \tau_{n} - \int_{\torus} g \circ q_{\mb t} \, d \tau \,\right|
< \varepsilon/2.
\]

The triangle inequality now yields 
\[\left|  \int_{\torus^{n}}  g \, d \tau_{n} 
 \; - \; 
\int_{\torus^{m}} g \circ q_A \, d \tau_{m}
\right|
< \varepsilon\]
as desired.
This completes the proof in the case where $n > 1$.
\end{proof}


\begin{proof}[Proof for $n=1$.]
In this case we will prove something stronger: we will show that if~$g\in \mc C(\torus^1)$ and if 
$q_A \colon \torus^m \to \torus^1$ is a continuous homomorphism with $\mu\left(q_A\right) \ge M = 2$,
then
$$
 \int_{\torus^{1}} g \, d \tau
 \;  = \; 
\int_{\torus^{m}} g \circ q_{A} \ d \tau_{m}.
$$
Trivially the two integrals will be distance less than any $\varepsilon>0$.  Note: this equality of integrals is just a special case of Corollary~\ref{cor_mt} below.

Suppose $q_A \colon \torus^m \to \torus^1$ is a continuous homomorphism with~$\mu(A) = \mu\left(q_A\right) \ge 2$. Since~$\mu(A) \ne 1$, the row matrix $A$ is not the zero matrix, and so has rank $1=n$.
In particular,~$q_A \colon \torus^m \to \torus$ is surjective (Proposition~\ref{infinite_height_lemma}).

The main trick of this proof (for $n=1$) is to regard $g$ not as a function on $\torus^1$ as originally given, but as a function on~$\torus^2$ where $g$  simply ignores the second coordinate;
this sets us up to use the $n = 2$ case of the theorem which has already been established.
To make this precise, let $\tilde g$ be a function (almost everywhere) on $\torus^2= \torus \times \torus$
defined by the rule~$\tilde g(z_1, z_2) = g(z_1)$. 
In particular, using the Fubini-Tonelli theorem,
\begin{eqnarray*}
\int_{\torus^2} \tilde g \ d\tau_2 
&=& \int_{\torus^2} g (z_1) \ d\tau_2(z_1, z_2) \\
&=& \int_{\torus} \left( \int_{\torus} g(z_1) \ d\tau(z_2) \right) d\tau  (z_1) \\
&=& \int_{\torus}  g(z_1) \left( \int_{\torus}  \ d\tau \right) d\tau  (z_1)   = \int_{\torus} g \ d\tau.
\end{eqnarray*}
We also claim that $\tilde g$ is in $\mc C(\torus^2)$, where $\mc C$ is the  Boyd--Lawton collection given in the statement of the theorem.
To see this, note that  $\tilde g = g\circ q_{\mb t}$ on $\torus^2$ where $\mb t = [1, 0]$ and $q_{\mb t} \colon \torus^2 \to \torus$ is the projection map to the first coordinate. Since $q_{\mb t}$ is surjective, we have 
that $\mu \left(q_{\mb t}\right) = \infty$ (Proposition~\ref{infinite_height_lemma}), which forces $\tilde g= g\circ q_{\mb t}$ to be in $\mc C(\torus^2)$ by the definition of Boyd--Lawton collection (Definition~\ref{BL collection def}, part 3).

Next, we expand $q_A$ to a map $\torus^{m+1} \to \torus^2$ by having it map as the identity on the last coordinate. More precisely,
consider the homomorphism $q_{\tilde A} \colon \torus^m \times \torus \to \torus \times \torus$ given by
$$
q_{\tilde A} (\mb z, w) = \left( q_A(\mb z), w \right).
$$
Note that $q_{\tilde A}$ is continuous and,
since $q_A$ is surjective, $q_{\tilde A}$ must also be surjective.  This implies that~$\mu\left( q_{\tilde A} \right) = \infty$ (Proposition~\ref{infinite_height_lemma}).
Using the  Fubini-Tonelli theorem as before we have
\begin{eqnarray*}
\int_{\torus^{m+1}} \tilde g \circ q_{\tilde A} \ d\tau_{m+1} 
&=& \int_{\torus^m} \left(  \int_{\torus} \tilde g \circ q_{\tilde A} (\mb z, w) \ d\tau (w) \right) \ d\tau_{m} (\mb z) \\
&=& \int_{\torus^m} \left(  \int_{\torus} \tilde g  \left(q_A(\mb z), w\right) \ d\tau (w) \right) \ d\tau_{m} (\mb z) \\
&=& \int_{\torus^m} \left(  \int_{\torus} g  \left(q_A(\mb z)\right) \ d\tau (w) \right) \ d\tau_{m} (\mb z) \\
&=& \int_{\torus^m} g  \left(q_A(\mb z)\right)  \left( \int_{\torus}  \ d\tau \right) d\tau_m  (\mb z)   = \int_{\torus^m} g \circ q_A \ d\tau_m.
\end{eqnarray*}

Finally we note that we have established the theorem for $n>1$, so it can be applied to the function $\tilde g \in \mc C(\torus^2)$.
In fact, since   $\mu\left( q_{\tilde A} \right) = \infty$, we get the stronger result that
$$
 \int_{\torus^{2}} \tilde g \, d \tau_{n}
 \;  = \; 
\int_{\torus^{m+1}} \tilde g \circ q_{\tilde A} \ d \tau_{m+1},
$$
as highlighted in the corollary below.
By substitution, we have
\[
 \int_{\torus^{1}} g \, d \tau
 \;  = \; 
\int_{\torus^{m}} g \circ q_{A} \ d \tau_{m}. \qedhere\]
\end{proof} 

If $A$ has infinite Boyd height, then the above gives an equality:

\begin{corollary}  \label{cor_mt}
Let $\mc C$ be a Boyd--Lawton collection and let $g \in \mc C_n$.
Suppose $q_A \colon \torus^m \to \torus^n$ is a nonzero continuous homomorphism of infinite Boyd height. In other words, suppose 
that~$q_A \colon \torus^m \to \torus^n$ is surjective.
Then
$$
 \int_{\torus^{n}} g \, d \tau_{n}
 \;  = \; 
\int_{\torus^m} g \circ q_A \, d \tau_m \,.
$$
\end{corollary}

\begin{remark}
Compare the above with Lemma~\ref{invariance_lemma2}. From this lemma, we see that the corollary holds for more general $g$, at least if $A$ is invertible.
\end{remark}


\section{Consequences}

Theorem~\ref{main_theorem} applies to any function $g$ in any Boyd--Lawton collection, including those
in the collections~$\mc C_1, \mc C_2, \mc C_3,$ and $\mc C_4$ described in Section~\ref{section_BL_Collections}.
If we combine Theorem~\ref{main_theorem} with Propositions~\ref{prop: C_1}, \ref{prop: C_2}, and \ref{prop: C_3}, 
then we obtain the following corollaries:

\begin{corollary}
If $g\colon\torus^n\to\bC$ is continuous,
then
$$
\int_{\torus^{n}} g \, d \tau_{n} \; = \;  \lim_{\mu\left( q_A \right)\to \infty}\;  \int_{\torus^m} g \circ q_A \; d \tau_m,
$$
where $q_A\colon \torus^m \to\torus^n$ varies among continuous homomorphisms. Here $m$ may vary in the limit, but $n$ is fixed by the choice of~$g$. 
\end{corollary}

\begin{corollary}
If $P \in\CC\left[Z_1^{\pm1},\dots, Z_n^{\pm1}\right]$ is a nonzero Laurent polynomial, then 
its Mahler measure satisfies the following:
$$
 \lim_{\mu\left({A}\right) \to\infty} m_1\left(P^{(A)}\right) = m_n(P),
$$
where in the limit expression $A$ varies among $n$ by $m$ integer matrices. Here $m$ may vary in the limit, but $n$ is fixed by the choice of~$P$. 
\end{corollary}

\begin{remark}
The above corollary is not new: it is essentially Theorem 1.1 and Theorem 3.1 of  \cite{Brunault2024}. However, our approach gives quite a different proof of this result.
\end{remark}

Finally, we get a generalization of Theorem~1.2 of \cite{Issa2013}, which to our knowledge is novel:

\begin{corollary} \label{extendedIssa2}
Let $P_1, \ldots, P_k \in\CC\left[Z_1^{\pm1},\dots, Z_n^{\pm1}\right]$ be nonzero Laurent polynomials. Then
the generalized Mahler measure satisfies
    \[\lim_{\mu\left( {A} \right)  \to \infty} m_{n, \max}\left( P_1^{(A)}, \ldots,  P_k^{(A)}\right)
    =m_{n, \max}( P_1, \ldots,  P_k),
    \]
    and the multiple higher Mahler measure satisfies
    \[
    \lim_{\mu\left( A \right)  \to \infty} m_{n, \operatorname{prod}}\left( P_1^{(A)}, \ldots,  P_k^{(A)}\right)
    =m_{n, \operatorname{prod}}(P_1, \ldots, P_k),
    \]
where in the limit expression $A$ varies among $n$ by $m$ integer matrices. Here $m$ may vary in the limit, but $n$ is fixed by the choice of the~$P_i$. 
\end{corollary}


\section*{Appendix: Homomorphisms between Real Tori (Part 2)}

In Section~\ref{Section_Homomorphisms} we made the claim that every continuous homomorphism $\varphi \colon \torus^m \to \torus^n$ is of the form $q_A$ for a unique $n$ by $m$ integer matrix $A$.
We believe this result is known, but we have yet to find a convenient reference, so we give details for the benefit of the reader. 

We start with the commutative diagram
 $$
 \begin{tikzcd} [column sep = large]
\bR^m  \arrow[r, dashed]  \arrow[d] \arrow[rd, "f"] & \bR^n \arrow[d]    \\
\bR^m/\bZ^m  \arrow[r, "\bar \varphi"]  \arrow[d, "\mathrm{exp}"] & \bR^n/\bZ^n  \arrow[d, "\mathrm{exp}"]   \\
\torus^m  \arrow[r,  "\varphi"]  & \torus^n    
\end{tikzcd}
$$
Here $\varphi$ is the given continuous homomorphism. 
The exponential map $\bR^m / \bZ^m \to T^m$ is the isomorphism 
$$
(z_1, \ldots, z_m) \mapsto \left( \exp(2 \pi z_1 i), \ldots, \exp(2 \pi z_m i) \right)
$$
and the exponential map $\bC^n / \bZ^n \to \torus^n$ is defined similarly.
These exponential maps are homeomorphisms and isomorphisms of groups, so there is a unique continuous homomorphism $\bar \varphi$ making the bottom square commute.
The unlabeled maps above are just the canonical projections, and $f$ is the projection $\bR^m \to \bR^m/\bZ^m$ followed by $\overline \varphi$.
Our goal is to study the liftings $\bR^m \to \bR^n$ of $f$ given above by a dashed arrow.

We first note that $\bR^n$ is a covering space of $\bR^n/\bZ^n$ with the canonical map giving a covering map.
See Munkres \S 53 \cite{munkres} for the definitions:  $\bR^n \to \bR^n/\bZ^n$ is a covering map since it is surjective and continuous, and every point of $\bR^n/\bZ^n$ has a neighborhood $U$
whose preimage is a disjoint countable union of neighborhoods $V_i$ such that the restriction  maps $V_i \to U$ are all homeomorphisms.
Furthermore, $\bR^m, \bR^n, \bR^n/\bZ^n$ are all path connected and are  all manifolds, so are all locally path connected. Finally, $\bR^m$ is simply connected with trivial fundamental group.
We can then apply Lemma 79.1 of \cite{munkres} to prove the following (using the fact that~$f(0) = 0$):

\begin{lemma}
Let $f$ be as above. Then there is a unique continuous map $\tilde \varphi \colon \bR^m \to \bR^n$ mapping $0$ to $0$ and making the following commute:
 $$
 \begin{tikzcd} [column sep = large]
\bR^m  \arrow[r, dashed, "\tilde \varphi"]  \arrow[d] \arrow[rd, "f"] & \bR^n \arrow[d]    \\
\bR^m/\bZ^m  \arrow[r, "\bar \varphi"] & \bR^n/\bZ^n 
\end{tikzcd}
$$
\end{lemma}

The uniqueness claim of this lemma is enough to show that there is at most one linear map $\tilde A$ such that $\varphi = q_A$.
What we need to do is show that the $\tilde \varphi$ of this lemma is in fact a linear map $\tilde A$. Suppose that we fix~$b \in \bR^m$ and consider the following function that measures
any failure of $\tilde \varphi$ to be a homomorphism:
$$
\delta (x) = \tilde \varphi (x+b) - \tilde\varphi(x) - \tilde\varphi(b).
$$
Observe that $\delta$ is continuous and $\delta(0) = 0$.
Consider $\pi \circ \delta$ where $\pi \colon \bR^n \to \bR^n/\bZ^n$ is the canonical map.
Since $\pi$ is a homomorphism and $\pi \circ \tilde \varphi = f$ is a homomorphism we see that 
$$
\pi \circ \delta (x) = \pi \circ \tilde \varphi  (x+b) - \pi \circ \tilde \varphi (x) - \pi \circ \tilde \varphi   (b) = f (x+b) - f (x) - f  (b) = 0.
$$
In particular, $\delta(x)$ is always in the kernel of $\pi$. In other words, $\delta(x) \in \bZ^n$. However, $\delta$ is continuous and $\delta(0) = 0$. Since $\bR^m$ is connected, 
$\delta(x) = 0$ for all $\bR^m$. We have established the following:

\begin{lemma}
Let $\tilde \varphi \colon \bR^m \to \bR^n$ be as in the above lemma. Then for all $x, b \in \bR^m$
$$
\tilde\varphi(x+b) = \tilde\varphi(x) + \tilde\varphi(b).
$$
In particular, $\tilde\varphi$ is a continuous homomorphism.
\end{lemma}

We can go further:

\begin{lemma}
Let $L\colon \bR^m \to \bR^n$ be a continuous homomorphism.
Then $L$ is actually a linear transformation.
\end{lemma}

\begin{proof}
For any positive integer $p$ and $x \in \bR^m$,
$$
L(p x) = L(x+ \ldots + x) = L(x) + \ldots + L(x) = p L(x).
$$
Replacing $x$ with $\frac 1 p y$ and multiplying the terms of the above equation by $\frac 1 p$ gives
$$
L\left ( \frac 1 p \cdot y\right) = \frac 1 p L(y).
$$
For any positive rational number $r = p/q$ with $p, q \in \bZ_{>0}$, and $y \in \bR^m$ we combine
the above laws to get
$$
L (r x) = L \left( \frac p q \cdot x \right) = p L  \left( \frac 1 q \cdot x \right) = \frac p q L(x) = r L(x).
$$
We can extend this to any $r \in \bQ$ since $L(0) = 0$, and 
$$L\big((-r) x\big) = L\big(-(rx)\big) = - L(r x) = - \big( r L(x) \big) = (-r) L(x)$$ for positive $r \in \bQ$.

Now for $r\in\bR$ write $r$ as the limit of $(r_i)$ with $r_i \in \bQ$. By continuity of $L$ and of the scalar product,
$$
L( r x) =L \left(  \lim_{i \to \infty}  r_i x \right) = \lim_{i \to \infty}  L(r_i x) = \lim_{i \to \infty}  r_i L(x)  = r L(x).
$$
Thus $L$  is a homomorphism compatible with scalar multiplication. In other words, $L$ is a linear transformation.
\end{proof}

In particular, the unique lifting $\tilde \varphi$ is not just a homomorphism, but must be a linear transformation $\tilde A$ for some matrix $A$.
Let $e_i$ be the $i$th element of the standard basis of $\bR^m$ and observe that $\tilde A(e_i)$ must be in $\bZ^n$.
This follows from the commutivity of the map
 $$
 \begin{tikzcd} [column sep = large]
\bR^m  \arrow[r, dashed, "\tilde A"]  \arrow[d] & \bR^n \arrow[d]    \\
\bR^m/\bZ^m  \arrow[r, "\bar \varphi"] & \bR^n/\bZ^n 
\end{tikzcd}
$$
and the fact that the image of $e_i$ in the group $\bR^m / \bZ^m$ is zero, so $\tilde A(e_i)$ is in the kernel of the homomorphism~$\bR^n \to \bR^n/\bZ^n$.
Since $\tilde A(e_i)$ is in $\bZ^n$, the $i$th column of the associated matrix~$A$ is in $\bZ^n$. In particular, $A$ is an integer matrix. 

Finally, since $\tilde \varphi = \tilde A$, the  corresponding maps in the following commutative diagram are equal:
 $$
 \begin{tikzcd} [column sep = large]
\bR^m  \arrow[r, "\tilde\varphi \, = \,  \tilde A"]  \arrow[d] & \bR^n \arrow[d]    \\
\bR^m/\bZ^m  \arrow[r, "\bar \varphi \, = \,  \bar A "]  \arrow[d, "\mathrm{exp}"] & \bR^n/\bZ^n  \arrow[d, "\mathrm{exp}"]   \\
\torus^m  \arrow[r,  "\varphi \, = \,  q_A"]  & \torus^n    
\end{tikzcd}
$$

In summary:
\begin{proposition}
Every continuous homomorphism $\varphi \colon \torus^m \to \torus^n$ is of the form $q_A$ for a unique $n$ by $m$ integer matrix $A$.
\end{proposition}

\bibliography{dyn2024}
\bibliographystyle{plainnat} 

\end{document}